\documentclass[11pt]{amsart}
\usepackage[T1]{fontenc}
\usepackage[utf8]{inputenc}
\usepackage{lmodern}
\usepackage[margin=28mm,top=27mm,bottom=27mm]{geometry}
\usepackage{amsmath,amssymb,amsthm,mathtools}
\usepackage{microtype}
\usepackage{tikz}
\usetikzlibrary{calc,arrows.meta}
\usepackage[colorlinks=true,linkcolor=blue!45!black,citecolor=blue!45!black,urlcolor=blue!45!black]{hyperref}
\hypersetup{pdftitle={Infinite light rays and infinite clusters with infinitely many pivots},pdfauthor={Martin P.W. Zerner},pdfsubject={Percolation and the Lorentz mirror model}}
\newtheorem{theorem}{Theorem}
\newtheorem{proposition}[theorem]{Proposition}
\newtheorem{lemma}[theorem]{Lemma}
\newtheorem{corollary}[theorem]{Corollary}
\theoremstyle{definition}
\newtheorem{definition}[theorem]{Definition}
\newtheorem{example}[theorem]{Example}
\newtheorem*{question}{Question}
\newcommand{\N}{\mathbb N}
\newcommand{\Z}{\mathbb Z}
\newcommand{\R}{\mathbb R}
\newcommand{\PP}{\mathbb P}
\newcommand{\EE}{\mathbb E}
\newcommand{\TT}{\mathbb T}
\newcommand{\bk}{\mathbf k}
\newcommand{\bK}{\mathbf K}
\newcommand{\ee}{\mathbf e}
\newcommand{\ff}{\mathbf f}
\newcommand{\bl}{\mathrm{bl}}
\newcommand{\wh}{\mathrm{wh}}
\newcommand{\crossstate}{\mathrm{cr}}
\DeclareMathOperator{\Bin}{Bin}

\title[Infinitely many pivots]{Infinite light rays and infinite clusters\\ with infinitely many pivots}
\author{Martin P.W. Zerner}
\address{University of T\"ubingen, Germany}
\email{martin.zerner@uni-tuebingen.de}
\date{October 2026}
\subjclass[2020]{60K35; 82B43, 05C05, 05C10}
\keywords{Lorentz mirror model, percolation at criticality, pivotal edge, medial graph, one-ended tree, uniform spanning tree}

\tikzset{vtx/.style={circle,fill=black,inner sep=0pt,minimum size=2.5pt},
strand/.style={black,line width=.65pt},mirror/.style={black,line width=1.5pt}}

\def\Amp{0.22}\def\Sp{0.5}\def\Ph{35.9}\def\Rl{0.9}

\newcommand{\cbraid}[4]{  \draw[#4,line width=1.2pt] plot[domain=#2:#3,samples=140,smooth,variable=\u]
    ({#1+\Amp*sin(180*\u)},{-0.2-\Sp*\u});
  \draw[#4,line width=1.2pt] plot[domain=#2:#3,samples=140,smooth,variable=\u]
    ({#1-\Amp*sin(180*\u)},{-0.2-\Sp*\u});}

\newcommand{\ctopL}[2]{\draw[#2,line width=1.2pt] (#1,-0.2) to[out=180-\Ph,in=0] (#1-0.4,0) -- (#1-\Rl,0);}
\newcommand{\ctopR}[2]{\draw[#2,line width=1.2pt] (#1,-0.2) to[out=\Ph,in=180] (#1+0.4,0) -- (#1+\Rl,0);}
\newcommand{\cbotL}[2]{\draw[#2,line width=1.2pt] (#1,-2.2) to[out=\Ph-180,in=0] (#1-0.4,-2.4) -- (#1-\Rl,-2.4);}
\newcommand{\cbotR}[2]{\draw[#2,line width=1.2pt] (#1,-2.2) to[out=-\Ph,in=180] (#1+0.4,-2.4) -- (#1+\Rl,-2.4);}
\newcommand{\ctop}[2]{\ctopL{#1}{#2}\ctopR{#1}{#2}}
\newcommand{\cbot}[2]{\cbotL{#1}{#2}\cbotR{#1}{#2}}
\newcommand{\hmir}[2]{\draw[mirror] ({#1-0.19},{-0.2-\Sp*#2})--({#1+0.19},{-0.2-\Sp*#2});}
\newcommand{\vmir}[2]{\draw[mirror] (#1,{-0.2-\Sp*#2-0.19})--(#1,{-0.2-\Sp*#2+0.19});}
\newcommand{\sseg}[4]{  \draw[#4,line width=1.2pt] plot[domain=#2:#2+1,samples=70,smooth,variable=\u]
    ({#1+#3*\Amp*abs(sin(180*\u))},{-0.2-\Sp*\u});}
\newcommand{\bdots}[1]{\foreach \j in {0,...,4}{\node[vtx] at (#1,{-0.2-\Sp*\j}){};}}

\providecommand{\E}{}\providecommand{\Ey}{}
\newcommand{\Ew}{0.22}          %
\newcommand{\Lb}{.7*.75*\Ew}     %
\newcommand{\ustedges}{%
\E{0}{0}{0}{1}
\E{0}{1}{90}{3}
\E{0}{2}{0}{3}
\E{0}{3}{0}{1}
\E{1}{0}{90}{2}
\E{1}{1}{180}{2}
\E{1}{2}{0}{3}
\E{1}{3}{0}{2}
\E{2}{0}{90}{3}
\Ey{2}{1}{90}{4}
\Ey{2}{2}{90}{4}
\Ey{2}{3}{0}{4}
\E{3}{0}{180}{3}
\Ey{3}{1}{180}{1}
\E{3}{2}{90}{1}
\Ey{3}{3}{0}{4}
\E{4}{0}{180}{3}
\E{4}{1}{90}{1}
\E{4}{2}{90}{2}
\Ey{4}{3}{0}{4}
\E{5}{0}{180}{3}
\E{5}{1}{0}{1}
\E{5}{2}{0}{1}
\Ey{5}{3}{0}{4}
\E{6}{0}{180}{2}
\E{6}{1}{270}{2}
\E{6}{2}{90}{2}
\Ey{6}{3}{90}{4}
}
\newcommand{\ustcorners}{%
\C{0}{0}{0}{1}{0}{1}
\C{0}{1}{0}{2}{90}{3}
\C{0}{1}{90}{3}{0}{2}
\C{0}{2}{0}{3}{270}{3}
\C{0}{2}{270}{3}{0}{3}
\C{0}{3}{0}{1}{0}{1}
\C{1}{0}{90}{2}{180}{1}
\C{1}{0}{180}{1}{90}{2}
\C{1}{1}{180}{2}{270}{2}
\C{1}{1}{270}{2}{180}{2}
\C{1}{2}{0}{3}{180}{3}
\C{1}{2}{180}{3}{0}{3}
\C{1}{3}{0}{2}{180}{1}
\C{1}{3}{180}{1}{0}{2}
\C{2}{0}{0}{3}{90}{3}
\C{2}{0}{90}{3}{0}{3}
\C{2}{1}{0}{1}{90}{4}
\C{2}{1}{90}{4}{270}{3}
\C{2}{1}{270}{3}{0}{1}
\C{2}{2}{90}{4}{180}{3}
\C{2}{2}{180}{3}{270}{4}
\C{2}{2}{270}{4}{90}{4}
\C{2}{3}{0}{4}{180}{2}
\C{2}{3}{180}{2}{270}{4}
\C{2}{3}{270}{4}{0}{4}
\C{3}{0}{0}{3}{180}{3}
\C{3}{0}{180}{3}{0}{3}
\C{3}{1}{180}{1}{180}{1}
\C{3}{2}{90}{1}{90}{1}
\C{3}{3}{0}{4}{180}{4}
\C{3}{3}{180}{4}{270}{1}
\C{3}{3}{270}{1}{0}{4}
\C{4}{0}{0}{3}{180}{3}
\C{4}{0}{180}{3}{0}{3}
\C{4}{1}{90}{1}{90}{1}
\C{4}{2}{90}{2}{270}{1}
\C{4}{2}{270}{1}{90}{2}
\C{4}{3}{0}{4}{180}{4}
\C{4}{3}{180}{4}{270}{2}
\C{4}{3}{270}{2}{0}{4}
\C{5}{0}{0}{2}{180}{3}
\C{5}{0}{180}{3}{0}{2}
\C{5}{1}{0}{1}{0}{1}
\C{5}{2}{0}{1}{0}{1}
\C{5}{3}{0}{4}{180}{4}
\C{5}{3}{180}{4}{0}{4}
\C{6}{0}{90}{2}{180}{2}
\C{6}{0}{180}{2}{90}{2}
\C{6}{1}{180}{1}{270}{2}
\C{6}{1}{270}{2}{180}{1}
\C{6}{2}{90}{2}{180}{1}
\C{6}{2}{180}{1}{90}{2}
\C{6}{3}{90}{4}{180}{4}
\C{6}{3}{180}{4}{270}{2}
\C{6}{3}{270}{2}{90}{4}
}

\newcommand{\Tbundle}[5][bun]{\foreach \j in {1,...,#5}{%
  \pgfmathsetmacro\o{(\j-(#5+1)/2)*\Ew}%
  \draw[#1] (#2,#3) .. controls ($(#2,#3)+(#4:.34)+({#4+90}:\o)$)
     and ($(#2,#3)+(#4:.66)+({#4+90}:\o)$) .. ($(#2,#3)+(#4:1)$);}}
\newcommand{\Tbraid}[5][med]{%
  \foreach \j in {1,...,#5}{\pgfmathsetmacro\o{.75*(\j-(#5+1)/2)*\Ew}%
    \node[mv] at ($(#2,#3)+(#4:.5)+({#4+90}:\o)$){};}%
  \ifnum#5>1 \foreach \j in {1,...,\numexpr#5-1\relax}{%
    \pgfmathsetmacro\o{.75*(\j-(#5+1)/2)*\Ew}\pgfmathsetmacro\p{\o+.75*\Ew}%
    \coordinate (ma) at ($(#2,#3)+(#4:.5)+({#4+90}:\o)$);
    \coordinate (mb) at ($(#2,#3)+(#4:.5)+({#4+90}:\p)$);
    \draw[#1] (ma) .. controls ($(ma)+({#4+135}:\Lb)$) and ($(mb)+({#4+225}:\Lb)$) .. (mb);
    \draw[#1] (ma) .. controls ($(ma)+({#4+45}:\Lb)$) and ($(mb)+({#4-45}:\Lb)$) .. (mb);}\fi}
\newcommand{\C}[6]{%
  \pgfmathsetmacro\ha{.75*(#4-1)*\Ew/2}\pgfmathsetmacro\hb{.75*(#6-1)*\Ew/2}%
  \pgfmathsetmacro\D{mod(#5-#3+360,360)}\pgfmathsetmacro\D{\D<1?360:\D}%
  \pgfmathsetmacro\s{0.2+0.0014*\D}\pgfmathsetmacro\t{\D>180?(\D-180)/180*0.45:0}%
  \coordinate (pa) at ($(#1,#2)+(#3:.5)+({#3+90}:\ha)$);
  \coordinate (pb) at ($(#1,#2)+(#5:.5)+({#5+90}:-\hb)$);
  \draw[med] (pa) .. controls ($(pa)+({#3+135}:\s)+({#3+180}:\t)$)
                    and ($(pb)+({#5+225}:\s)+({#5+180}:\t)$) .. (pb);}

\newcommand{\Abfnt}[4]{\Tbundle[bfnt]{#1}{#2}{#3}{#4}}
\newcommand{\Abraid}[4]{\Tbraid[med]{#1}{#2}{#3}{#4}}
\newcommand{\Abraidray}[4]{\begin{scope}[mv/.style={circle,fill=orange!85!black,inner sep=0pt,minimum size=2.4pt}]%
  \Tbraid[medray]{#1}{#2}{#3}{#4}\end{scope}}

\newcommand{\bundlefaces}[2]{%
  \pgfmathtruncatemacro{\lastgap}{#2-2}%
  \foreach \j in {0,...,\lastgap}{%
    \pgfmathsetmacro{\upper}{.36*(#2-1)-.72*\j}%
    \foreach \side in {-1,1}{%
      \draw[medial] plot[domain=0:1,samples=31,variable=\vv]
        ({2*#1+1+\side*.27*sin(180*\vv)},
         {(\upper-.72*\vv)*(1-(.27*sin(180*\vv))^2)});}}}

\newcommand{\bundleeyes}[2]{%
  \pgfmathtruncatemacro{\lastgap}{#2-2}%
  \foreach \j in {0,...,\lastgap}{%
    \pgfmathsetmacro{\upper}{.36*(#2-1)-.72*\j}%
    \fill[white] plot[domain=0:1,samples=31,variable=\vv]
        ({2*#1+1-.27*sin(180*\vv)},{(\upper-.72*\vv)*(1-(.27*sin(180*\vv))^2)})
      -- plot[domain=1:0,samples=31,variable=\vv]
        ({2*#1+1+.27*sin(180*\vv)},{(\upper-.72*\vv)*(1-(.27*sin(180*\vv))^2)})
      -- cycle;}}

\begin{document}
\begin{abstract}
It is conjectured that in the Lorentz mirror model on $\Z^2$ all trajectories are finite. Quas showed that, almost surely, the two halves of any infinite trajectory, cut at an edge, meet at infinitely many vertices, each of which is pivotal: changing its state can make the trajectory finite. We construct plane graphs on which infinite trajectories exist, and all of them have this property. Given $p\in(0,1)$, we replace each edge of a random one-ended subtree of $\Z^2$ by a number of parallel edges, depending on $p$, that is at most logarithmic in the height of the corresponding descendant tree. The medial graph of the resulting plane multigraph is a factor of the tree, and for the uniform spanning tree it has finite vertex intensity. In the mirror model on this medial graph, with probability $p$ for mirrors along primal edges and any probability $r\in(0,1-p]$ for mirrors along dual edges, infinite trajectories exist, and the two halves of each of them meet infinitely often. The construction rests on Bernoulli percolation: every random one-ended locally finite tree can be decorated in this way such that bond percolation with parameter $p$ is critical and percolates, and every vertex of the infinite cluster has infinitely many pivotal edges. We characterize this property for bond and site percolation and for mirrors.
\end{abstract}
\maketitle

\section{Introduction}
In the Lorentz mirror model on $\Z^2$ with parameter $p\in(0,\frac12]$, independently at each vertex a two-sided mirror is placed along one of the two diagonals, each with probability $p$, and no mirror with probability $1-2p$. Light travels along the edges of $\Z^2$, is reflected by the mirrors and passes straight through empty vertices. The edges traversed by the light through an edge $\ee$ form its trajectory $C_\ee$. For $p=\frac12$, all trajectories are almost surely finite, and it is conjectured that the same holds for every $p\in(0,\frac12)$, see \cite{G,G2}. Quas \cite[Theorem~3]{Q} showed that for $0<p<\frac12$, almost surely, whenever $C_\ee$ is infinite, the two halves of $C_\ee$ obtained by cutting at $\ee$ meet at infinitely many vertices. By Lemma~\ref{lem:meeting}, changing the state of any of these vertices can make $C_\ee$ finite. %
Graphs on which infinite trajectories exist and all are of this kind are called $p$-fragile for mirrors (Definition~\ref{def:mirrors}), and by Quas' theorem the conjecture is equivalent to $\Z^2$ not being $p$-fragile for mirrors for any $p\in(0,\frac12)$.

One might hope that this meeting property alone excludes infinite trajectories. Our main result shows that it does not, even on random plane graphs of finite intensity whose laws are mixing and invariant under all symmetries of $\Z^2$. None of our results depends on \cite{Q}.

Let $G$ be a connected multigraph embedded in $\R^2$ such that every compact set meets only finitely many edges. Its medial graph $M(G)$ is a 4-regular plane multigraph with a vertex at the midpoint of each edge of $G$, see Section~\ref{sec:mirrors} and Figure~\ref{fig:chain}. In the mirror model on $M(G)$ with parameters $p,r>0$, $p+r\le1$, independently at the midpoint of each edge $e$ of $G$, a two-sided mirror is placed along $e$ with probability $p$, along the dual edge $e^*$ with probability $r$, and no mirror with probability $1-p-r$. If $G$ is the square lattice rotated by $45^\circ$ and scaled by $\sqrt2$ such that $M(G)=\Z^2$, then $r=p$ gives the Lorentz mirror model. If $r=1-p$, the trajectories are the interfaces between the open clusters of $G$ and those of its dual in Bernoulli bond percolation on $G$ with parameter $p$.

A ray is a one-sided infinite simple path, and a tree is one-ended if any two of its rays have a common tail. A subtree of $\Z^d$ is a tree whose vertices are points of $\Z^d$ and whose edges are edges of the lattice $\Z^d$; it need not be spanning. For a one-ended subtree $\TT$ of $\Z^2$, the ancestor $A(x)$ of a vertex $x$ of $\TT$ is its neighbor on the ray from $x$ to infinity, and $H(x)$ is the height of the tree of descendants of $x$, see \eqref{eq:height}. For $\bK=(K_x)_{x\in V(\TT)}\in\N^{V(\TT)}$, let $\TT_\bK$ be obtained from $\TT$ by replacing each edge $\{x,A(x)\}$ by $K_x$ parallel copies. We draw them close to this lattice edge, in the same way for all lattice edges and symmetrically with respect to the edge and to its perpendicular bisector.

Random subtrees $\TT$ of $\Z^d$ are defined on a probability space with probability measure $\PP$ and expectation $\EE$. A factor of $\TT$ is a random object of the form $\Phi(\TT)$, where $\Phi$ is a measurable map that commutes with all symmetries $\sigma$ of $\Z^d$, i.e.\ $\Phi(\sigma T)=\sigma\Phi(T)$ for all subtrees $T$ of $\Z^d$. The law of a factor of $\TT$ is invariant under every symmetry of $\Z^d$ that preserves the law of $\TT$, and it is ergodic or mixing under translations if the law of $\TT$ is.

\begin{theorem}\label{th:main}
Let $\TT$ be a random one-ended subtree of $\Z^2$ and $0<p<1$. There is a function $g:\N_0\times\N\to\N$ with $g(a,b)\le c_p\log(a+2)$, where $c_p<\infty$ depends only on $p$, such that for $K_x:=g(H(x),H(A(x)))$ the following holds. The medial graph $M(\TT_\bK)$ is a factor of $\TT$. For almost every realization of $\TT$, every $r\in(0,1-p]$ and every edge $\ee$ of $M(\TT_\bK)$, the trajectory $C_\ee$ is infinite with positive probability, and almost surely on this event the two halves of $C_\ee$ meet at infinitely many vertices. If $\TT$ is stationary and $\EE[\log(H(0)+2);\,0\in V(\TT)]<\infty$, then the expected number of vertices of $M(\TT_\bK)$ in a unit square is finite.
\end{theorem}

The simplest instance is the deterministic subtree $\N_0\times\{0\}$ of $\Z^2$, for which $M(\TT_\bK)$ is a chain of braids as in Figure~\ref{fig:chain}. Our main example is the uniform spanning tree of $\Z^2$. Its law is invariant under all symmetries of $\Z^2$ and mixing \cite[Section~8]{BLPS}, and $\PP[H(0)\ge n]=n^{-3/5+o(1)}$ \cite{LLS}, see Section~\ref{sec:trees}. Thus the meeting property of Quas' theorem, even together with these properties and finite intensity, does not exclude infinite trajectories.

The construction is based on percolation. A graph is called $p$-fragile for bonds if Bernoulli bond percolation with parameter $p$ has an infinite cluster, and almost surely every vertex of an infinite cluster has infinitely many pivotal edges, whose closure would make its cluster finite. Proposition~\ref{prop:perc} characterizes this property and its analogue for sites, for instance by the absence of bi-infinite open trails, or by the structure of the infinite cluster as a one-ended tree of finite blocks. It implies $p=p_c$, i.e.\ percolation at criticality. By Theorem~\ref{th:tree}, every random one-ended locally finite tree becomes $p$-fragile for bonds when its edges are replaced by bundles as in Theorem~\ref{th:main}. The medial graph of such a bundled tree is a tree of braids, and for every $r\in(0,1-p]$, it is $(p,r)$-fragile for mirrors if and only if the bundled tree is $p$-fragile for bonds (Proposition~\ref{prop:tree-mirrors}).

Throughout, graphs are connected and locally finite and may have parallel edges. Assertions about percolation or mirrors on random graphs refer to the product measure given the graph, for almost every graph.

\section{Fragility for percolation}\label{sec:perc}
Let $G$ be infinite and $0<p<1$. Under $P_p$, the edges in bond percolation, or the vertices in site percolation, are independently open with probability $p$. Let $C_x$ be the open cluster of $x$, with $C_x=\varnothing$ if $x$ is closed in site percolation, $I_x:=\{|C_x|=\infty\}$, and $p_c(G):=\inf\{u:P_u[I_x]>0\}$. An edge or vertex is pivotal for an event if changing its state changes whether the event occurs. Let $\Pi_x$ be the set of pivots for $I_x$, without $x$ in site percolation.

\begin{lemma}[Finite energy; cf.\ {\cite[Coupling~2.5]{HJ}}]\label{lem:energy}
Let $P$ be a product measure on a countable product of finite sets. If every configuration in an event $A$ can be changed at finitely many coordinates, to states of positive probability, such that an event $B$ occurs, then $P[A]>0$ implies $P[B]>0$.
\end{lemma}

In bond percolation, the blocks of an open cluster are the components that remain after deleting its bridges, and contracting them gives its bridge-block tree. In site percolation, the blocks are the maximal connected subgraphs without a cutvertex, and the block-cutvertex tree has a vertex for each block and each cutvertex, with edges between cutvertices and the blocks containing them.

\begin{proposition}\label{prop:perc}
Suppose $P_p[I_x]>0$ for some, hence every, $x\in V(G)$. In bond percolation the following are equivalent:
\begin{align*}
\text{(A)}\quad&\forall x:\quad P_p[|\Pi_x|=\infty\mid I_x]=1;\\
\text{(E)}\quad&\exists x:\quad P_p[|\Pi_x|=\infty\mid I_x]=1;\\
\text{(A$'$)}\quad&\forall x:\quad P_p[|\Pi_x|\ge1\mid I_x]=1;\\
\text{(B)}\quad&\text{almost surely there is no bi-infinite open trail;}\\
\text{(BB)}\quad&\text{almost surely every infinite open cluster has only finite blocks}\\
&\text{and a one-ended bridge-block tree.}
\end{align*}
Here a trail is a path without repeated edges. The same holds in site percolation, with ``trail'' replaced by ``simple path'' and ``bridge-block tree'' by ``block-cutvertex tree''. In either case, these conditions imply $p=p_c(G)$, and that almost surely there is exactly one infinite open cluster.
\end{proposition}

\begin{definition}\label{def:perc}
If the hypothesis and the equivalent conditions of Proposition~\ref{prop:perc} hold, then $G$ is called \emph{$p$-fragile for bonds}, respectively \emph{for sites}.
\end{definition}

By (BB), every infinite open cluster is a one-ended tree of finite blocks. Seen from a fixed vertex $x$, the blocks on the ray from a block containing $x$ to the end form an infinite chain of finite ``sausages'', and the rest of the cluster consists of finite pieces attached to this chain. The same picture appears in Menshikov's proof of exponential decay below $p_c$, where the cluster of the origin conditioned to reach distance $n$ ``resembles a chain of sausages'' \cite[Section~5.2]{G}, \cite{Me}, and in the lace expansion at criticality \cite{HS}. Conditions like (B), on disjoint infinite open paths, also appear for incipient infinite clusters, which arise by conditioning and taking limits, see e.g.\ \cite{CFHJ}. On $p$-fragile graphs, (B) holds together with $P_p[I_x]>0$ under $P_p$ itself.

\begin{proof}[Proof of Proposition~\ref{prop:perc}]
Write $P=P_p$. An element is an edge in bond and a vertex in site percolation, and $S(\gamma)$ denotes the set of elements of a path $\gamma$. Opening a finite path from $x$ to $y$ and applying Lemma~\ref{lem:energy} shows that $P[I_x]>0$ implies $P[I_y]>0$. Clearly (A) implies (E) and (A$'$).

$\neg$(A)$\Rightarrow\neg$(E): Let $P[A_x]>0$, where $A_x:=I_x\cap\{|\Pi_x|<\infty\}$, and fix $y$. On $A_x$, open all elements of a finite $x$--$y$ path $\gamma$. This preserves $I_x$ and creates no new pivots for $I_x$. Outside $\gamma$, closing an element leaves $x$ and $y$ connected, so that their pivotalities agree. Hence $\Pi_y\subseteq\Pi_x\cup S(\gamma)$ after the change, and $P[A_y]>0$ by Lemma~\ref{lem:energy}.

$\neg$(A)$\Rightarrow\neg$(A$'$): On $A_x$, the finitely many pivots for $I_x$ lie on every open ray from $x$. Choose such a ray and a vertex $y$ on it beyond these pivots. A pivot for $I_y$ would lie on the ray beyond $y$, and closing it would leave $x$ connected to $y$ and destroy $I_x$, so that it would belong to $\Pi_x$. Hence $\Pi_y=\varnothing$, and by countability $P[I_y,\Pi_y=\varnothing]>0$ for some fixed $y$.

$\neg$(B)$\Rightarrow\neg$(A$'$): By countability, with positive probability a fixed vertex $x$ lies on a bi-infinite open trail, respectively simple path. Closing an edge, respectively a vertex other than $x$, leaves intact one of the two halves of this trail, respectively path, starting at $x$. Hence $P[I_x,\Pi_x=\varnothing]>0$.

$\neg$(A$'$)$\Rightarrow\neg$(B): Let $P[I_x,\Pi_x=\varnothing]>0$, and denote by $H$ the cluster $C_x$ on this event. For $n\ge1$, take the ball of radius $n$ about $x$ in $H$ together with its outer neighbors in $H$, and identify the latter to a single vertex $t$, keeping parallel edges. In this finite graph, $x$ and $t$ are not adjacent, and no edge, respectively vertex other than $x$ and $t$, separates them, since closing it would confine $C_x$ to the ball. By Menger's theorem \cite[Corollary~3.3.5]{D}, there are two $x$--$t$ paths that are edge-disjoint, respectively disjoint apart from $x$ and $t$. Delete $t$ and truncate both paths after $n$ edges. These pairs of paths, for $n\ge1$, and their simultaneous prefixes form a locally finite tree with vertices at every depth. By K\"onig's infinity lemma \cite[Lemma~8.1.2]{D}, it contains a ray, which yields two open rays from $x$ that are edge-disjoint, respectively disjoint apart from $x$. Together they form a bi-infinite open trail, respectively simple path.

(A$'$)$\Rightarrow$(BB): By countability, almost surely $\Pi_x\ne\varnothing$ for every vertex $x$ of every infinite open cluster $H$. An element of $\Pi_x$ is a bridge, respectively a cutvertex other than $x$, of $H$ whose deletion leaves $x$ in a finite component. Hence every block of $H$ containing $x$ is finite. As $x$ is arbitrary, the bridge-block, respectively block-cutvertex, tree of $H$ is infinite and locally finite, so that it has an end. It has only one: a vertex $x$ in a block on the double ray between two ends would remain connected to one of them after the deletion of any bridge, respectively of any cutvertex other than $x$, whence $\Pi_x=\varnothing$.

(BB)$\Rightarrow$(A): On $I_x$, the deletion of any of the infinitely many bridges, respectively cutvertices other than $x$, on the ray from a block containing $x$ to the end leaves $x$ in a component consisting of finitely many finite blocks. All of them are therefore pivots for $I_x$.

$p=p_c(G)$: For $0<u<p$, thin the $p$-open elements independently with retention probability $u/p$. Under (A), $I_x$ survives the thinning only if all its infinitely many pivots do, which has conditional probability $0$. Hence $P_u[I_x]=0$.

Existence and uniqueness: The existence of an infinite open cluster is invariant under changes of finitely many elements, and therefore has probability $1$ by Kolmogorov's zero-one law. If two infinite clusters occurred with positive probability, opening a shortest path between two rays in different infinite clusters would create a bi-infinite open simple path. By Lemma~\ref{lem:energy}, this contradicts (B).
\end{proof}

\section{Bundles on one-ended trees}\label{sec:trees}
Fix $0<p<1$ and write $q=1-p$. Let $T$ be a one-ended locally finite tree. The ancestor $a(x)$ of $x\in V(T)$ is its neighbor on the ray from $x$ to infinity. The tree of descendants of $x$, formed by all $y$ with $a^n(y)=x$ for some $n\ge0$, is finite, since otherwise it would contain a ray, which would give a second end. For $\bk=(k_x)_{x\in V(T)}\in\N^{V(T)}$, let $T_\bk$ be obtained from $T$ by replacing each edge $\{x,a(x)\}$ by $k_x$ parallel edges, and let $\bk^x:=(k_{a^n(x)})_{n\ge0}$. For the tree $\N_0$ with edges $\{n,n+1\}$, $n\ge0$, the ancestor of $n$ is $n+1$, and $(\N_0)_\bk$ has $k_n$ parallel edges between $n$ and $n+1$, see Figure~\ref{fig:chain}.

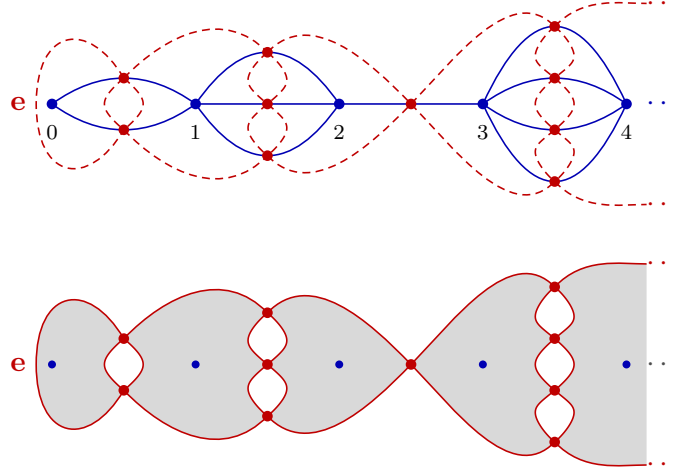
\begin{figure}[tb]
\centering
\begin{tikzpicture}[x=.95cm,y=.95cm,
  original/.style={blue!70!black,line width=.6pt},
  medial/.style={red!75!black,densely dashed,line width=.6pt},
  original vertex/.style={circle,fill=blue!70!black,inner sep=0pt,minimum size=4pt},
  medial vertex/.style={circle,fill=red!75!black,inner sep=0pt,minimum size=4pt}]
\foreach \n/\k in {0/2,1/3,2/1,3/4}{%
  \pgfmathtruncatemacro{\lastedge}{\k-1}%
  \foreach \j in {0,...,\lastedge}{%
    \pgfmathsetmacro{\h}{.36*(\k-1)-.72*\j}%
    \draw[original] (2*\n,0) .. controls ({2*\n+2/3},{4*\h/3}) and ({2*\n+4/3},{4*\h/3}) .. ({2*\n+2},0);}}
\foreach \n in {0,...,4}{\node[original vertex] at (2*\n,0) {};\node[below=4pt,font=\scriptsize] at (2*\n,0) {$\n$};}
\node[blue!70!black] at (8.55,0) {$\cdots$};
\bundlefaces{0}{2}\bundlefaces{1}{3}\bundlefaces{3}{4}
\draw[medial] (1,.36) .. controls (.5,1.2) and (-.22,1.05) .. (-.22,0) .. controls (-.22,-1.05) and (.5,-1.2) .. (1,-.36);
\node[red!75!black,left] at (-.22,0) {$\ee$};
\foreach \n/\h/\hh in {0/.36/.72,1/.72/0,2/0/1.08}{\foreach \s in {-1,1}{%
  \draw[medial] ({2*\n+1},{\s*\h}) .. controls ({2*\n+1.6},{\s*(\h+.62)}) and ({2*\n+2.4},{\s*(\hh+.62)}) .. ({2*\n+3},{\s*\hh});}}
\foreach \s in {-1,1}{\draw[medial] (7,{\s*1.08}) .. controls (7.4,{\s*1.5}) and (7.9,{\s*1.4}) .. (8.28,{\s*1.4});
  \node[red!75!black] at (8.55,{\s*1.4}) {$\cdots$};}
\foreach \n/\k in {0/2,1/3,2/1,3/4}{\pgfmathtruncatemacro{\lastedge}{\k-1}%
  \foreach \j in {0,...,\lastedge}{\node[medial vertex] at ({2*\n+1},{.36*(\k-1)-.72*\j}) {};}}
\end{tikzpicture}

\medskip
\begin{tikzpicture}[x=.95cm,y=.95cm,
  medial/.style={red!75!black,line width=.6pt},
  original vertex/.style={circle,fill=blue!70!black,inner sep=0pt,minimum size=3pt},
  medial vertex/.style={circle,fill=red!75!black,inner sep=0pt,minimum size=4pt}]
\fill[gray!30] (1,.36) .. controls (1.6,.98) and (2.4,1.34) .. (3,.72)
  .. controls (3.6,1.34) and (4.4,.62) .. (5,0)
  .. controls (5.6,.62) and (6.4,1.70) .. (7,1.08)
  .. controls (7.4,1.5) and (7.9,1.4) .. (8.28,1.4) -- (8.28,-1.4)
  .. controls (7.9,-1.4) and (7.4,-1.5) .. (7,-1.08)
  .. controls (6.4,-1.70) and (5.6,-.62) .. (5,0)
  .. controls (4.4,-.62) and (3.6,-1.34) .. (3,-.72)
  .. controls (2.4,-1.34) and (1.6,-.98) .. (1,-.36)
  .. controls (.5,-1.2) and (-.22,-1.05) .. (-.22,0)
  .. controls (-.22,1.05) and (.5,1.2) .. (1,.36) -- cycle;
\bundleeyes{0}{2}\bundleeyes{1}{3}\bundleeyes{3}{4}
\bundlefaces{0}{2}\bundlefaces{1}{3}\bundlefaces{3}{4}
\draw[medial] (1,.36) .. controls (.5,1.2) and (-.22,1.05) .. (-.22,0) .. controls (-.22,-1.05) and (.5,-1.2) .. (1,-.36);
\node[red!75!black,left] at (-.22,0) {$\ee$};
\foreach \n/\h/\hh in {0/.36/.72,1/.72/0,2/0/1.08}{\foreach \s in {-1,1}{%
  \draw[medial] ({2*\n+1},{\s*\h}) .. controls ({2*\n+1.6},{\s*(\h+.62)}) and ({2*\n+2.4},{\s*(\hh+.62)}) .. ({2*\n+3},{\s*\hh});}}
\foreach \s in {-1,1}{\draw[medial] (7,{\s*1.08}) .. controls (7.4,{\s*1.5}) and (7.9,{\s*1.4}) .. (8.28,{\s*1.4});
  \node[red!75!black] at (8.55,{\s*1.4}) {$\cdots$};}
\node[gray!60!black] at (8.55,0) {$\cdots$};
\foreach \n in {0,...,4}{\node[original vertex] at (2*\n,0) {};}
\foreach \n/\k in {0/2,1/3,2/1,3/4}{\pgfmathtruncatemacro{\lastedge}{\k-1}%
  \foreach \j in {0,...,\lastedge}{\node[medial vertex] at ({2*\n+1},{.36*(\k-1)-.72*\j}) {};}}
\end{tikzpicture}
\caption{Top: $(\N_0)_\bk$ for $\bk=(2,3,1,4,\ldots)$ (blue) and its medial graph $M((\N_0)_\bk)$ (red). The edge $\ee$ lies in the outer corner at the vertex $0$ of $(\N_0)_\bk$. Bottom: $M((\N_0)_\bk)$ with its black faces shaded; they contain the vertices of $(\N_0)_\bk$.}
\label{fig:chain}
\end{figure}

\begin{proposition}\label{prop:bundle}
The multigraph $(\N_0)_\bk$ is $p$-fragile for bonds if and only if
\begin{equation}\label{eq:series}
\sum_{n\ge0}q^{k_n}<\infty\qquad\text{and}\qquad\sum_{n\ge0}k_nq^{k_n}=\infty.
\end{equation}
\end{proposition}
\begin{proof}
The numbers $Y_n$ of open edges between $n$ and $n+1$ are independent with law $\Bin(k_n,p)$. Hence
\begin{equation}\label{eq:product}
P_p[I_0]=\prod_{n\ge0}(1-q^{k_n})>0\quad\Longleftrightarrow\quad\sum_{n\ge0}q^{k_n}<\infty.
\end{equation}
On $I_0$, $|\Pi_0|=\sum_n\mathbf 1_{\{Y_n=1\}}$, and $P_p[Y_n=1]=k_npq^{k_n-1}$. The Borel--Cantelli lemmas and Proposition~\ref{prop:perc} complete the proof.
\end{proof}

For $p=\frac12$, the ``if'' part is due to Lyons and Schramm \cite[p.~1816]{LS}, as a variation of an example of H\"aggstr\"om and Peres: for blocks of $k_n$ independent fair bits, the event that some block consists of ones only has, with positive probability, infinitely many pivotal bits. Interchanging zeros and ones and passing to the complement turns this event into $I_0$. Angel and Hutchcroft \cite[Proposition~3.2]{AH} use bundles on a canopy tree to obtain percolation at criticality, see Example~\ref{ex:canopy}.

\begin{example}\label{ex:log}
For $\gamma>1$, let
\begin{equation}\label{eq:log}
k_n=\bigl\lceil\log_{1/q}\bigl(n(\log n)^\gamma\bigr)\bigr\rceil\qquad\text{for all sufficiently large }n.
\end{equation}
Then $q^{k_n}$ is bounded from above and below by positive multiples of $1/(n(\log n)^\gamma)$, and $k_n$ is of order $\log n$. Hence \eqref{eq:series} holds if and only if $\gamma\le2$. Since moreover $\sum_n(1-u)^{k_n}=\infty$ for $u<p$, \eqref{eq:product} shows that $(\N_0)_\bk$ percolates at $p=p_c((\N_0)_\bk)$. For $\gamma=2$, the remaining $k_n$ can be chosen such that $k_n\le c_p\log(n+2)$ for all $n$.
\end{example}

\begin{proposition}\label{prop:det-tree}
Let $\bk\in\N^{V(T)}$ and $x\in V(T)$. Then $T_\bk$ is $p$-fragile for bonds if and only if $(\N_0)_{\bk^x}$ is.
\end{proposition}
\begin{proof}
The branches off the ray $(a^n(x))_{n\ge0}$ consist of descendants of its vertices and are therefore finite. Deleting them, together with their bundles, leaves a copy of $(\N_0)_{\bk^x}$, with $x$ in place of $0$, and preserves $I_x$ and $\Pi_x$. By Proposition~\ref{prop:perc}, each of the two multigraphs is $p$-fragile for bonds if and only if $P_p[I_x]>0$ and $P_p[|\Pi_x|=\infty\mid I_x]=1$.
\end{proof}

\begin{example}\label{ex:canopy}
The $d$-ary canopy tree, $d\ge2$, has levels $0,1,\ldots$; each vertex at level $n\ge1$ has $d$ neighbors at level $n-1$ and one at level $n+1$, and the vertices at level $0$ are leaves. Replace each edge between the levels $n$ and $n+1$ by $k_n$ parallel edges, with $k_n$ as in \eqref{eq:log} for $q=\frac12$ and some $\gamma>1$. Angel and Hutchcroft \cite[Proposition~3.2]{AH} showed that this unimodular random graph has $p_c=\frac12$ and percolates at $p_c$. For $x$ at level $0$, $\bk^x=(k_n)_{n\ge0}$. By Example~\ref{ex:log} and Propositions~\ref{prop:bundle} and~\ref{prop:det-tree}, the graph is $\frac12$-fragile for bonds if and only if $\gamma\le2$.
\end{example}

Now let $(\TT,\rho)$ be a random rooted one-ended locally finite tree, defined on a probability space with probability measure $\PP$. Denote its ancestor map by $A$, and let
\begin{equation}\label{eq:height}
H(x):=\max\{n\ge0:\ A^n(y)=x\text{ for some }y\in V(\TT)\}
\end{equation}
be the height of the tree of descendants of $x$. Then $H(x)<\infty$ and $H(A(x))>H(x)$.

\begin{theorem}\label{th:tree}
Let $0<p<1$. There is a function $g:\N_0\times\N\to\N$, depending only on $p$ and on the laws of $H(A^n(\rho))$, $n\ge0$, such that $\TT_\bK$ with
\begin{equation}\label{eq:K}
K_x:=g\bigl(H(x),H(A(x))\bigr)
\end{equation}
is almost surely $p$-fragile for bonds, and $g(a,b)\le c_p\log(a+2)$ for all $a,b$, so that
\begin{equation}\label{eq:bound}
K_x\le c_p\log\bigl(H(x)+2\bigr)\qquad\text{for all }x\in V(\TT),
\end{equation}
where $c_p<\infty$ depends only on $p$.
\end{theorem}

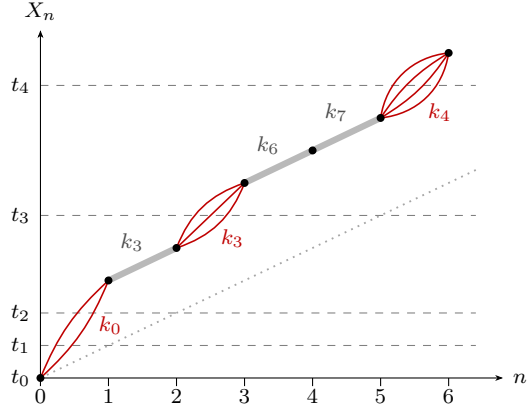
\begin{figure}[tb]
\centering
\begin{tikzpicture}[x=9mm,y=4.3mm,
  kb/.style={red!75!black,line width=.6pt},
  kg/.style={gray!55,line width=2.4pt},
  thr/.style={gray,dashed,line width=.4pt},
  dia/.style={gray!70,dotted,line width=.7pt},
  sv/.style={circle,fill=black,inner sep=0pt,minimum size=3pt}]
\foreach \y/\i in {1/1,2/2,5/3,9/4}{
  \draw[thr] (0,\y)--(6.4,\y);
  \node[left,font=\scriptsize] at (0,\y) {$t_\i$};}
\node[left,font=\scriptsize] at (0,0) {$t_0$};
\foreach \n in {0,...,6}{\draw (\n,0)--(\n,-.25) node[below,font=\scriptsize,inner sep=1pt] {$\n$};}
\draw[dia] (0,0)--(6.4,6.4);
\draw[-{Stealth[length=3pt]}] (0,0)--(6.8,0) node[right,font=\scriptsize] {$n$};
\draw[-{Stealth[length=3pt]}] (0,0)--(0,10.7) node[above,font=\scriptsize] {$X_n$};
\foreach \a/\b/\c/\d/\l in {1/3/2/4/3,3/6/4/7/6,4/7/5/8/7}{
  \draw[kg] (\a,\b)--(\c,\d);
  \node[font=\scriptsize,gray!60!black] at ($(\a,\b)!.5!(\c,\d)!3.2mm!90:(\c,\d)$) {$k_\l$};}
\foreach \a/\b/\c/\d/\i/\m/\s/\o in {0/0/1/3/0/2/.7/3.6mm,2/4/3/6/3/3/.5/4.2mm,5/8/6/10/4/4/.5/4.8mm}{
  \foreach \j in {1,...,\m}{
    \pgfmathsetmacro{\bnd}{(\j-(\m+1)/2)*25}
    \draw[kb] (\a,\b) to[bend left=\bnd] (\c,\d);}
  \node[font=\scriptsize,red!75!black] at ($(\a,\b)!\s!(\c,\d)!\o!-90:(\c,\d)$) {$k_\i$};}
\foreach \a/\b in {0/0,1/3,2/4,3/6,4/7,5/8,6/10}{\node[sv] at (\a,\b) {};}
\end{tikzpicture}
\caption{The bundles along the ray from $\rho$ in the proof of Theorem~\ref{th:tree}, drawn schematically; the bundle between $A^n(\rho)$ and $A^{n+1}(\rho)$ is drawn from $(n,X_n)$ to $(n+1,X_{n+1})$. Red: $k_i$ edges, where $i$ is the smallest index with $t_i\in[X_n,X_{n+1})$. Gray: $k_{X_n}$ edges if there is no such $i$. Dotted: the diagonal. Here $t_0,t_1,t_2\in[X_0,X_1)$, so that $k_1$ and $k_2$ do not occur, and $k_3$ occurs twice.}
\label{fig:staircase}
\end{figure}

\begin{proof}
Choose $(k_i)_{i\ge0}$ as at the end of Example~\ref{ex:log}, and let $X_n:=H(A^n(\rho))$. Then $(X_n)_{n\ge0}$ is strictly increasing, and $X_n\ge n$. Choose integers $0=t_0<t_1<\cdots$ recursively such that
\begin{equation}\label{eq:threshold}
\PP[X_{t_{i-1}+1}>t_i]\le2^{-i}\qquad(i\ge1).
\end{equation}
Let $f(a,b)$ be the smallest $i\ge0$ with $a\le t_i<b$, and $f(a,b):=a$ if there is no such $i$. Set $g(a,b):=k_{f(a,b)}$. Then $K_{A^n(\rho)}=k_{f(X_n,X_{n+1})}$, see Figure~\ref{fig:staircase}.

Since the intervals $[X_n,X_{n+1})$ are disjoint and the $X_n$ are distinct, every integer occurs at most twice among the values $f(X_n,X_{n+1})$, $n\ge0$. Hence
\begin{equation}\label{eq:sum-upper}
\sum_{n\ge0}q^{K_{A^n(\rho)}}\le2\sum_{i\ge0}q^{k_i}<\infty.
\end{equation}
By \eqref{eq:threshold} and the Borel--Cantelli lemma, almost surely $X_{t_{i-1}+1}\le t_i$ for all sufficiently large $i$. For such $i$, let $N_i$ be the largest $n$ with $X_n\le t_i$. Then $N_i>t_{i-1}$ and $X_{N_i}\le t_i<X_{N_i+1}$. As $t_j\le t_{i-1}<N_i\le X_{N_i}$ for all $j<i$, it follows that $f(X_{N_i},X_{N_i+1})=i$. In particular, the $N_i$ are distinct, and
\[
\sum_{n\ge0}K_{A^n(\rho)}q^{K_{A^n(\rho)}}\ge\sum_{i\text{ large}}k_iq^{k_i}=\infty.
\]
Propositions~\ref{prop:bundle} and~\ref{prop:det-tree} with $x:=\rho$ yield fragility. If $f(a,b)=i\ge1$ with $a\le t_i<b$, then $t_{i-1}<a$ by the minimality of $i$, so that $i-1\le t_{i-1}<a$. Hence $f(a,b)\le a$, and $g(a,b)\le c_p\log(f(a,b)+2)\le c_p\log(a+2)$.
\end{proof}

\begin{corollary}\label{cor:stationary}
Let $\TT$ be a random one-ended subtree of $\Z^d$, let $\rho$ be its first vertex in a fixed enumeration of $\Z^d$, and let $g$ and $\bK$ be as in Theorem~\ref{th:tree}. Then $\TT_\bK$ is a factor of $\TT$, and $\EE[K_0;\,0\in V(\TT)]<\infty$ if $\EE[\log(H(0)+2);\,0\in V(\TT)]<\infty$. If $\TT$ is stationary, then $\EE[\deg_{\TT_\bK}(0);\,0\in V(\TT)]=2\EE[K_0;\,0\in V(\TT)]$.
\end{corollary}
\begin{proof}
The maps $A$ and $H$ commute with all symmetries of $\Z^d$, and $g$ is deterministic. The moment bound follows from \eqref{eq:bound}. Under stationarity, with $V:=V(\TT)$,
\begin{align*}
\EE\bigl[\deg_{\TT_\bK}(0);\,0\in V\bigr]&=\EE[K_0;\,0\in V]+\sum_{y\sim0}\EE\bigl[K_y;\,y\in V,\,A(y)=0\bigr]\\
&=\EE[K_0;\,0\in V]+\sum_{y\sim0}\EE\bigl[K_0;\,0\in V,\,A(0)=-y\bigr]=2\EE[K_0;\,0\in V].\qedhere
\end{align*}
\end{proof}

Every stationary one-ended spanning tree of $\Z^2$ satisfies $\liminf_nn\PP[H(0)\ge n]>0$ \cite[Theorem~1]{BZZ}, so that $\EE[H(0)]=\infty$. Finite expected degree therefore requires multiplicities that grow much more slowly than $H$, as in \eqref{eq:bound}. The uniform spanning tree of $\Z^2$ is one-ended by \cite{Pe} and \cite[Theorem~10.1]{BLPS}, its law is invariant under all symmetries of $\Z^2$ and mixing \cite[Section~8]{BLPS}, and $\PP[H(0)\ge n]=n^{-3/5+o(1)}$ by \cite[Remark~3.6 and Table~1]{LLS}. In particular, $\EE[\log(H(0)+2)]<\infty$. The north-east tree, in which each vertex independently chooses its upper or its right neighbor as its ancestor, is a stationary one-ended spanning tree of $\Z^2$, and $\PP[H(0)\ge n]\asymp n^{-1/2}$ follows easily from the fact that the paths to infinity in this tree and in its dual tree are coalescing random walks \cite{ZM}.

\section{Mirrors on plane 4-regular graphs}\label{sec:mirrors}
Let $M$ be a connected 4-regular multigraph embedded in $\R^2$ such that every compact set meets only finitely many edges. Its faces can be colored black and white such that every edge separates a black face from a white one, because a closed curve that crosses the edges transversally encloses finitely many vertices, all of even degree, and therefore crosses an even number of edges. We fix such a coloring, which is unique up to interchanging the colors. At each vertex, the four corners between consecutive half-edges are alternately black and white.

There are three ways to pair the four half-edges at a vertex, cf.\ \cite[Section~17.8]{GR}. In the state $\bl$, the half-edges are paired around the two white corners, in the state $\wh$ around the two black corners, and in the state $\crossstate$ opposite half-edges are paired. The states $\bl$ and $\wh$ are two-sided mirrors, and we call $\crossstate$ a crossing. For $p,r>0$ with $p+r\le1$, let $P_{p,r}$ be the product measure under which the vertices of $M$ are independently in the states $\bl$, $\wh$ and $\crossstate$ with probabilities $p$, $r$ and $1-p-r$. For $p=r$, the coloring plays no role.

Now let $G$ be a connected multigraph embedded in $\R^2$ such that every compact set meets only finitely many edges. Its medial graph $M(G)$ has a vertex $v_e$ at the midpoint of each edge $e$ of $G$. Each corner between consecutive edge ends at a vertex of $G$ gives an edge of $M(G)$, which joins the corresponding midpoints through this corner; at a vertex of degree one, this is a loop. Thus $M(G)$ is a 4-regular plane multigraph as above, and we color its faces black if they contain a vertex of $G$, and white otherwise, in which case they correspond to faces of $G$, see Figure~\ref{fig:chain}. Of the four corners at $v_e$, the two black ones contain the two halves of $e$, and the two white ones those of the dual edge $e^*$. Hence $\bl$ is a mirror along $e$ and $\wh$ a mirror along $e^*$, and declaring $e$ open if $v_e$ is in the state $\bl$ gives Bernoulli bond percolation on $G$ with parameter $p$. If $G$ is the square lattice formed by the centers of the black squares of a chessboard coloring of $\Z^2$, then $M(G)=\Z^2$, and $P_{p,p}$ is the Lorentz mirror model of Section~1.

Following the pairings partitions the edges of $M$ into trajectories, each of which is a finite closed trail or a bi-infinite trail. For an edge $\ee$ of $M$, let $C_\ee$ be its trajectory, $I_\ee:=\{|C_\ee|=\infty\}$, and let $\Pi_\ee$ be the set of vertices at which a change to another state of positive probability changes whether $I_\ee$ occurs. On $I_\ee$, removing $\ee$ splits $C_\ee$ into two halves.

\begin{lemma}\label{lem:meeting}
On $I_\ee$, a vertex belongs to $\Pi_\ee$ if and only if both halves of $C_\ee$ visit it. If $C_\ee=C_\ff$, then $|\Pi_\ee\triangle\Pi_\ff|<\infty$.
\end{lemma}
\begin{proof}
If at most one half visits $x$, a change at $x$ does not affect the other half. If both halves visit $x$, consider the state at $x$ that pairs the two half-edges through which they first arrive at $x$. It turns the segment of $C_\ee$ between these arrivals, which contains $\ee$, into a finite trajectory. If $p+r<1$, this state has positive probability. If $p+r=1$, there are no crossings, and this state is a mirror: otherwise the two arrivals are opposite at $x$, and the closed segment is a curve without transversal self-intersections that passes straight through $x$ and separates the other two half-edges at $x$. One of the two infinite tails of $C_\ee$ beyond $x$ would then start inside this curve, could neither cross it nor return to $x$, and would therefore be finite by local finiteness.

If $C_\ee=C_\ff$ is infinite, the halves for $\ee$ and those for $\ff$ differ only on the finite segment between $\ee$ and $\ff$. If $C_\ee$ is finite, then $\Pi_\ee$ and $\Pi_\ff$ consist of vertices of $C_\ee$.
\end{proof}

\begin{proposition}\label{prop:mirrors}
Suppose $P_{p,r}[I_\ee]>0$ for some edge $\ee$ of $M$. Then this holds for every edge, almost surely there is an infinite trajectory, and the following are equivalent:
\begin{align*}
\text{(A)}\quad&\forall\ee:\quad P_{p,r}[|\Pi_\ee|=\infty\mid I_\ee]=1;\\
\text{(E)}\quad&\exists\ee:\quad P_{p,r}[|\Pi_\ee|=\infty\mid I_\ee]=1;\\
\text{(A$'$)}\quad&\forall\ee:\quad P_{p,r}[|\Pi_\ee|\ge1\mid I_\ee]=1.
\end{align*}
\end{proposition}

\begin{definition}\label{def:mirrors}
If the hypothesis and the equivalent conditions of Proposition~\ref{prop:mirrors} hold, then $M$ is called \emph{$(p,r)$-fragile for mirrors}, and \emph{$p$-fragile for mirrors} if $r=p$.
\end{definition}

\begin{proof}[Proof of Proposition~\ref{prop:mirrors}]
Write $P=P_{p,r}$. Let $\ee$ and $\ff$ be edges with a common endpoint $x$. On $I_\ee\setminus I_\ff$, the trajectories $C_\ee$ and $C_\ff$ use disjoint pairs of half-edges at $x$, and changing the state at $x$ to another state of positive probability joins them into one infinite trajectory. Lemma~\ref{lem:energy} and induction along paths show $P[I_\ff]>0$ for all $\ff$. A change at a vertex $x$ does not affect the infinite trajectories avoiding $x$, and cuts an infinite trajectory through $x$ into pieces, at least one of which is an infinite one-sided trail; this trail lies on an infinite trajectory after the change. Applied also to the reverse change, this shows that the existence of an infinite trajectory is invariant under changes at finitely many vertices. Hence it has probability $1$ by Kolmogorov's zero-one law. Clearly (A) implies (E) and (A$'$).

$\neg$(A)$\Rightarrow\neg$(E): Let $P[A_\ee]>0$, where $A_\ee:=I_\ee\cap\{|\Pi_\ee|<\infty\}$, and let $\ff$ have a common endpoint $x$ with $\ee$. On $A_\ee$, we make $A_\ff$ occur by changing finitely many states to states of positive probability. If $C_\ff$ is finite, we change the state at $x$ as above. This inserts $C_\ff$ into $C_\ee$ as a finite excursion, which changes the set of vertices visited by both halves only at finitely many vertices, so that $A_\ff$ occurs by Lemma~\ref{lem:meeting}. If $I_\ff$ occurs and $|\Pi_\ff|=\infty$, then $C_\ff\ne C_\ee$ by Lemma~\ref{lem:meeting}. Changing the state at some $y\in\Pi_\ff$ makes $C_\ff$ finite. As $C_\ff$ uses all four half-edges at $y$, this does not affect $C_\ee$, and we are back in the previous case. By Lemma~\ref{lem:energy} and induction, $P[A_\ff]>0$ for all $\ff$.

$\neg$(A)$\Rightarrow\neg$(A$'$): On $A_\ee$, if $\Pi_\ee\ne\varnothing$, let $x$ be the last vertex of $\Pi_\ee$ on one fixed half of $C_\ee$, and change the state at $x$ so as to close the segment through $\ee$; this joins the two remaining infinite tails at $x$. For an edge $\ff$ at $x$ on the resulting infinite trajectory, $\Pi_\ff=\varnothing$: by Lemma~\ref{lem:meeting}, every $y\in\Pi_\ff$ would lie on both tails, hence belong to the original $\Pi_\ee$ beyond $x$. If $\Pi_\ee=\varnothing$, let $\ff:=\ee$. Lemma~\ref{lem:energy} and countability give $P[I_\ff,\Pi_\ff=\varnothing]>0$ for some fixed $\ff$.
\end{proof}

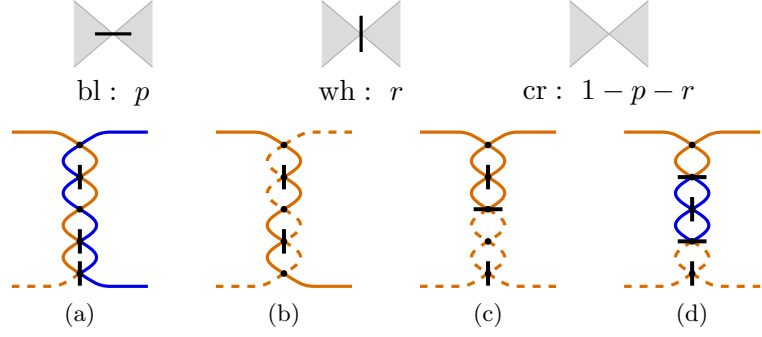
\begin{figure}[tb]
\centering
\begin{tikzpicture}[x=.8cm,y=.8cm]
\foreach \x in {0,4.1,8.2}{
 \fill[gray!28] (\x,0)--(\x-.65,.55)--(\x-.65,-.55)--cycle;
 \fill[gray!28] (\x,0)--(\x+.65,.55)--(\x+.65,-.55)--cycle;
 \draw[gray!65] (\x-.65,.55)--(\x+.65,-.55);
 \draw[gray!65] (\x-.65,-.55)--(\x+.65,.55);
}
\draw[very thick] (-.3,0)--(.3,0);
\draw[very thick] (4.1,-.3)--(4.1,.3);
\node[below] at (0,-.6){$\bl:\ p$};
\node[below] at (4.1,-.6){$\wh:\ r$};
\node[below] at (8.2,-.6){$\crossstate:\ 1-p-r$};
\end{tikzpicture}
\par\medskip
\begin{tikzpicture}[xscale=1.0,yscale=0.85]
 \cbraid{8.1}{0}{1}{orange!85!black}\ctop{8.1}{orange!85!black}
\cbraid{8.1}{1}{3}{blue!90!black}
 \cbraid{8.1}{3}{4}{orange!85!black,dashed}\cbot{8.1}{orange!85!black,dashed}
 \hmir{8.1}{1}\vmir{8.1}{2}\hmir{8.1}{3}\vmir{8.1}{4}
 \bdots{8.1}
 \node[font=\footnotesize,align=center] at (8.1,-2.85) {(d)};
 \cbraid{5.4}{0}{2}{orange!85!black}\ctop{5.4}{orange!85!black}
 \cbraid{5.4}{2}{4}{orange!85!black,dashed}\cbot{5.4}{orange!85!black,dashed}
 \vmir{5.4}{1}\hmir{5.4}{2}\vmir{5.4}{4}
 \bdots{5.4}
 \node[font=\footnotesize,align=center] at (5.4,-2.85) {(c)};
 \sseg{0}{0}{1}{orange!85!black}\sseg{0}{1}{1}{orange!85!black}
 \sseg{0}{2}{-1}{orange!85!black}\sseg{0}{3}{-1}{orange!85!black}
 \ctopL{0}{orange!85!black}\cbotR{0}{blue!90!black}
 \sseg{0}{3}{1}{blue!90!black}\sseg{0}{2}{1}{blue!90!black}
 \sseg{0}{1}{-1}{blue!90!black}\sseg{0}{0}{-1}{blue!90!black}
 \cbotL{0}{orange!85!black,dashed}\ctopR{0}{blue!90!black}
 \vmir{0}{1}\vmir{0}{3}\vmir{0}{4}
 \bdots{0}
 \node[font=\footnotesize,align=center] at (0,-2.85) {(a)};
 \sseg{2.7}{0}{1}{orange!85!black}\sseg{2.7}{1}{1}{orange!85!black}
 \sseg{2.7}{2}{-1}{orange!85!black}\sseg{2.7}{3}{-1}{orange!85!black}
 \ctopL{2.7}{orange!85!black}\cbotR{2.7}{orange!85!black}
 \sseg{2.7}{3}{1}{orange!85!black,dashed}\sseg{2.7}{2}{1}{orange!85!black,dashed}
 \sseg{2.7}{1}{-1}{orange!85!black,dashed}\sseg{2.7}{0}{-1}{orange!85!black,dashed}
 \cbotL{2.7}{orange!85!black,dashed}\ctopR{2.7}{orange!85!black,dashed}
 \vmir{2.7}{1}\vmir{2.7}{3}
 \bdots{2.7}
 \node[font=\footnotesize,align=center] at (2.7,-2.85) {(b)};
\end{tikzpicture}
\caption{Top: the states $\bl$, $\wh$ and $\crossstate$ at a vertex of a braid, with the black corners shaded. Bottom: the connections from the two left exits of a braid (solid and dashed orange) and other trajectories (blue). (a), (b): no black mirror, with an even, respectively odd, number of crossings. (c): exactly one black mirror. (d): two black mirrors.}
\label{fig:braids}
\end{figure}

The medial graph $M((\N_0)_\bk)$ is a chain of braids, see Figure~\ref{fig:chain}. The braid of the edges between $n$ and $n+1$ has $k_n$ vertices and four exits, two toward $n$ on the left and two toward $n+1$ on the right. At its vertices, the black corners face $n$ and $n+1$, so that black mirrors are horizontal and white mirrors vertical, see Figures~\ref{fig:chain} and~\ref{fig:braids}.

\begin{proposition}\label{prop:braids}
Let $\bk\in\N^{\N_0}$ and $r\in(0,1-p]$. Then $M((\N_0)_\bk)$ is $(p,r)$-fragile for mirrors if and only if $(\N_0)_\bk$ is $p$-fragile for bonds.
\end{proposition}
\begin{proof}
Let $\ee$ be the edge in the outer corner at $0$; it joins the two left exits of the first braid. If a braid has a black mirror, each of its left exits is connected to a right exit. These two connections meet only at the black mirror if there is exactly one, and are disjoint if there are more. If the braid has no black mirror, the two left exits are connected either to each other or to the opposite right exits, and in the latter case both connections visit all vertices of the braid. See Figure~\ref{fig:braids}. Without black mirrors, a white mirror keeps the side of the braid and a crossing changes it. Hence the left exits are connected to each other if and only if the number of crossings is even. With $s:=1-p-r=q-r$, the probabilities of these two cases for a braid with $k$ vertices are
\begin{equation}\label{eq:braid-exact}
b_k=\frac{q^k+(r-s)^k}{2}\qquad\text{and}\qquad d_k=\frac{q^k-(r-s)^k}{2}.
\end{equation}
Since $b_k/q^k=(1+(2r/q-1)^k)/2$ and $-1<2r/q-1\le1$,
\begin{equation}\label{eq:braid-bounds}
\min\{r/q,\tfrac12\}\,q^k\le b_k\le q^k\qquad\text{and}\qquad 0\le d_k\le q^k.
\end{equation}
By the independence of the braids,
\begin{equation}\label{eq:mirror-product}
P_{p,r}[I_\ee]=\prod_{n\ge0}(1-b_{k_n})>0\quad\Longleftrightarrow\quad\sum_{n\ge0}q^{k_n}<\infty.
\end{equation}
Let $\sum_nq^{k_n}<\infty$. By the Borel--Cantelli lemma, almost surely only finitely many braids have no black mirror. On $I_\ee$, all other meetings of the two halves are the black mirrors in the braids with exactly one black mirror, which occur independently with probabilities $k_npq^{k_n-1}$. The Borel--Cantelli lemmas, Lemma~\ref{lem:meeting} and Propositions~\ref{prop:bundle} and~\ref{prop:mirrors} complete the proof.
\end{proof}

\begin{figure}[tb]
\centering
\begin{tikzpicture}[scale=1.48,
  bfnt/.style={blue!50,line width=.45pt},
  med/.style={red!75!black,line width=.5pt},
  medray/.style={orange!85!black,line width=.8pt},
  mv/.style={circle,fill=red!75!black,inner sep=0pt,minimum size=2.4pt}]
\begin{scope}
\clip (-.6,-.6) rectangle (6.6,3.95);
\let\E\Abfnt \let\Ey\Abfnt \ustedges
\ustcorners
\let\E\Abraid \let\Ey\Abraidray \ustedges
\end{scope}
\draw[med] (5.7525,3.5) to[out=120,in=-60] (5.62,3.78);
\draw[med] (6.2475,3.5) to[out=60,in=240] (6.38,3.78);
\foreach \t in {1,2,3}{
 \fill[red!75!black] ($(5.62,3.78)+(120:.09*\t)$) circle (.65pt);
 \fill[red!75!black] ($(6.38,3.78)+(60:.09*\t)$) circle (.65pt);}
\node[font=\small,orange!85!black,inner sep=0pt] at (3.4,1) {$x$};
\end{tikzpicture}
\caption{A bundled one-ended tree (blue) and its medial graph (red and orange). The orange braids belong to the ray from the leaf $x$. Every finite branch off this ray gives a piece of the medial graph with two exits. Replacing these pieces by single edges leaves a chain of braids as in Figure~\ref{fig:chain}.}
\label{fig:tree-medial}
\end{figure}
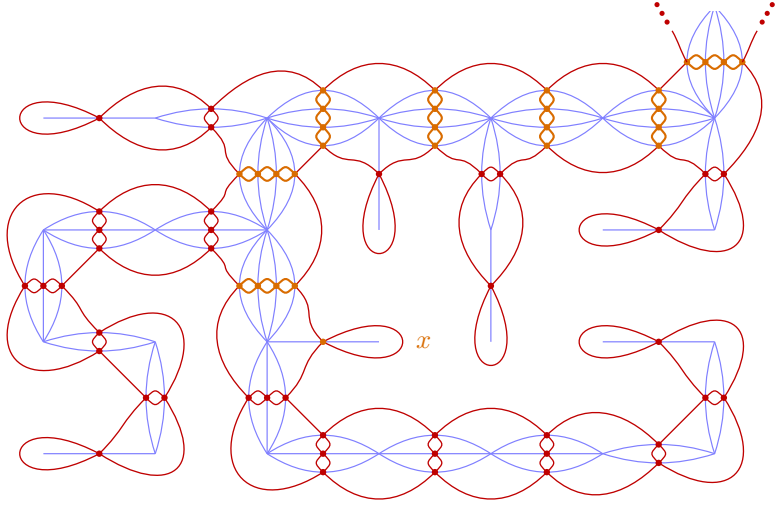

\begin{proposition}\label{prop:tree-mirrors}
Let $T$ be a one-ended locally finite tree embedded in $\R^2$ such that every compact set meets only finitely many edges, let $\bk\in\N^{V(T)}$, and draw the $k_x$ copies of each edge $\{x,a(x)\}$ close to it. Then, for every $r\in(0,1-p]$, $M(T_\bk)$ is $(p,r)$-fragile for mirrors if and only if $T_\bk$ is $p$-fragile for bonds.
\end{proposition}
\begin{proof}
Let $x$ be a leaf of $T$; it exists since the trees of descendants are finite. Each finite branch of $T_\bk$ off the ray $(a^n(x))_{n\ge0}$, together with the bundle joining it to the ray, gives a finite piece of $M(T_\bk)$ with two exits, see Figure~\ref{fig:tree-medial}. In every configuration, the pairings decompose such a piece into closed trails and one trail between its two exits. Replacing each piece by a single edge therefore preserves all trajectories outside the pieces, and yields $M((\N_0)_{\bk^x})$ with independent states with the same probabilities. No vertex inside a piece is visited by both halves of an infinite trajectory through an edge outside the piece. Hence, for the edge $\ee$ in the outer corner at $x$, the event $I_\ee$ occurs in $M(T_\bk)$ if and only if it occurs in $M((\N_0)_{\bk^x})$, and on $I_\ee$ the two sets $\Pi_\ee$ coincide by Lemma~\ref{lem:meeting}. By Proposition~\ref{prop:mirrors}, each of these two medial graphs is $(p,r)$-fragile if and only if $P_{p,r}[I_\ee]>0$ and $P_{p,r}[|\Pi_\ee|=\infty\mid I_\ee]=1$. Propositions~\ref{prop:braids} and~\ref{prop:det-tree} complete the proof.
\end{proof}

\begin{proof}[Proof of Theorem~\ref{th:main}]
Let $\rho$, $g$ and $\bK$ be as in Corollary~\ref{cor:stationary} with $d=2$. By Theorem~\ref{th:tree} and Proposition~\ref{prop:tree-mirrors}, almost surely $M(\TT_\bK)$ is $(p,r)$-fragile for mirrors for every $r\in(0,1-p]$. Proposition~\ref{prop:mirrors} and Lemma~\ref{lem:meeting} yield the assertions about the trajectories. The parallel copies and their medial vertices can be drawn in disjoint narrow neighborhoods of the lattice edges, and the medial edges in the corners in disjoint small neighborhoods of the lattice vertices, equivariantly under the symmetries of $\Z^2$. Since $\TT_\bK$ is a factor of $\TT$ by Corollary~\ref{cor:stationary}, so is $M(\TT_\bK)$. The bundle of $x$ carries $K_x$ vertices of $M(\TT_\bK)$. Hence, if $\TT$ is stationary, the vertex intensity of $M(\TT_\bK)$ is $\EE[K_0;\,0\in V(\TT)]$, which is finite by Corollary~\ref{cor:stationary} if $\EE[\log(H(0)+2);\,0\in V(\TT)]<\infty$.
\end{proof}

By Propositions~\ref{prop:perc} and~\ref{prop:tree-mirrors}, whether $M(T_\bk)$ is $(p,r)$-fragile for mirrors does not depend on $r$, and it requires $p=p_c(T_\bk)$. The proof of $p=p_c$ in Proposition~\ref{prop:perc} relies on the monotonicity of Bernoulli percolation in $p$, which has no counterpart for mirrors.

\begin{question}
Is there a connected 4-regular multigraph embedded in $\R^2$ such that every compact set meets only finitely many edges, which is $p$-fragile for mirrors for two different values of $p$?
\end{question}

\section*{Acknowledgment}
This paper is based on a talk given at the workshop \emph{Probability and Statistical Physics} on the occasion of Franz Merkl's 60th birthday, LMU Munich, October 9, 2026. The author thanks Elena Kosygina for helpful discussions. This preprint was written with substantial assistance from the AI systems ChatGPT (OpenAI) and Claude (Anthropic), which were used for drafting text, proofs and figures; the author is responsible for its content.

\end{document}